\documentclass[11pt]{amsart}
\usepackage{epsfig}
\usepackage{amscd}
\usepackage[mathscr]{eucal}
\usepackage{amssymb}
\usepackage{amsxtra}
\usepackage{amsmath}
\usepackage[all]{xy}
\usepackage{color}

\newtheorem{thm}{Theorem}

\newtheorem{rem}[thm]{Remark}

\newtheorem{lem}[thm]{Lemma}

\newtheorem{defn}[thm]{Definition}

\newtheorem{que-ed}{Question-editing}

\newtheorem{con}[thm] {Conjecture}

\def\PHameo{\operatorname{PHameo}}

\begin{document}

\title[Hamiltonian isotopies]{A note on $C^0$ rigidity of Hamiltonian isotopies}
\date{\today}
\author{Sobhan Seyfaddini }

\address{University of California Berkeley\\ Berkeley, CA 94720 \\USA}
\email{sobhan@math.berkeley.edu}

\begin{abstract}

{\noindent We show that a symplectic isotopy that is a $C^0$ limit of Hamiltonian isotopies is itself Hamiltonian, if the corresponding  sequence of generating Hamiltonians converge in $L^{(1, \infty)}$ topology.
}

\end{abstract}

\maketitle 

\section{Introduction}
   Let $(M,\omega)$ denote a $2n$-dimensional closed and connected symplectic manifold. Denote by $Symp(M)$ and $Ham(M)$ the groups of symplectomorphisms and Hamiltonian diffeomorphisms of $M$, respectively.  This note addresses a question related to the $C^0$-Flux conjecture.  
  
  Banyaga's foundational paper \cite{Ba} gave rise to the $C^0$-Flux conjecture; it states that $Ham(M)$ is $C^0$-closed in $Symp_0(M)$, the path component of identity in $Symp(M)$.  An article by Lalonde, McDuff, and Polterovich \cite{LMP} contains a full discussion of this conjecture and a proof of it in special cases.  As is noted in \cite{LMP}, the $C^0$-Flux conjecture is a fundamental question lying at the boundary between hard and soft symplectic topology; very little is known about the $C^0$-Flux conjecture and today's symplectic technology seems to be insufficient for approaching this question in full generality.  
  
  The reader may wonder why it is asked if $Ham(M)$ is $C^0$-closed in $Symp_0(M)$ rather than $Symp(M)$.  The difficulty in addressing the latter question is that, although symplectic rigidity tells us that $Symp(M)$ is $C^0$-closed in the group of diffeomorphisms of $M$,  it is not known if $Symp_0(M)$ is $C^0$-closed in $Symp(M)$.  To avoid this difficulty the $C^0$-Flux conjecture is usually stated for $Symp_0(M)$.  For further information about this intriguing conjecture we refer the reader to \cite{MS, LMP}.  
  
  In this note we address a path version of the $C^0$-Flux conjecture: we ask if the space of (based) Hamiltonian isotopies is $C^0$-closed in the space of (based) symplectic isotopies.  If the $C^0$-Flux conjecture were true, it would immediately answer this question in the affirmative.  Below, we will show that a symplectic isotopy which is a $C^0$ limit of Hamiltonian isotopies is itself Hamiltonian, provided that the corresponding  sequence of  generating Hamiltonians converge in $L^{(1, \infty)}$ topology.  In particular, because uniform convergence of functions is stronger than $L^{(1,\infty)}$ convergence, the same statement is true if we assume that the generating Hamiltonians converge uniformly.
  
  We should mention that the $C^1$-Flux conjecture, which states that $Ham(M)$ is $C^1$-closed in $Symp(M)$, was settled by Ono in \cite{On}.  Note that in this case there is no need for restricting to $Symp_0(M)$ because it is known that $Symp_0(M)$ is $C^1$-closed in $Symp(M)$.  This follows from the fact that a $C^1$-small neighborhood of identity in $Symp(M)$ can be identified with a neighborhood of zero in the space of closed $1$-forms on $M$ \cite{W}.
  
  \section{Main Result}
  In this section we set our notation and state our main result.

  Denote by $Homeo(M)$ and $Diff(M)$ the groups of homeomorphisms and diffeomorphisms of $M$.  A (smooth) isotopy of $M$ is a (smooth) continuous map $\phi: [0,1] \times M \rightarrow M$, such that the time-t map, obtained by fixing $t$ and denoted by $\phi^t$, is a homeomorphism (or diffeomorphism if $\phi$ is smooth) of $M$. Throughout this note, unless otherwise stated, we assume $\phi^0 = Id$. A smooth isotopy $\phi^t$ is called symplectic/Hamiltonian if $\phi^t \in Symp(M)/Ham(M)$ for all $t \in [0,1]$.  We denote by $PSymp(M)$ and $PHam(M)$ the set of all symplectic and Hamiltonians isotopies, respectively.  We use the notation $\phi^t_H$ to denote the Hamiltonian isotopy generated by the Hamiltonian $H: [0,1] \times M \rightarrow \mathbb{R}.$  Throughout this note we assume all Hamiltonians are normalized in the sense that  $\int_M H(t,x) \omega^n = 0, \forall t\in[0,1]$.  Hamiltonian isotopies are sometimes called Hamiltonian paths; we will be using both terminologies.
  
  Let $d$ denote a distance induced on $M$ by a Riemannian metric.  For $\phi, \; \psi \in Homeo(M)$ define their $C^0$ distance by $d_{C^0}(\phi, \psi) = \max_x d(\phi(x), \psi(x))$.  We define the $C^0$ distance between two isotopies $\phi^t$ and $\psi^t$ by $d_{C^0}^{path}(\phi^t, \psi^t) = \max_{0 \leq t \leq 1} d_{C^0}(\phi^t, \psi^t).$  The topologies induced by the above distances are referred to as $C^0$ topologies.  In this notation, the path version of the $C^0$-Flux conjecture can be stated as:
  
  \begin{con}\label{Main Question}
       $PHam(M)$ is $C^0$-closed in $PSymp(M)$.
   \end{con}
   
  Suppose that $\phi^t = (C^0) \lim_{i\to \infty} \phi^t_{H_i}$ and that $\phi^t$ is symplectic.  Conjecture \ref{Main Question} says that $\phi^t$ is a Hamiltonian path.  First, observe that, by the Gromov-Eliashberg $C^0$-rigidity theorem, requiring the isotopy $\phi^t$ to be symplectic is no stronger than requiring it to be smooth.   Second, in dimension 2, a simple argument can be used to show that $\phi^t$ has vanishing flux and hence it must be Hamiltonian.  However, this argument can not be generalized to all symplectic manifolds and hence the question remains open in higher dimensions.  The main result of this note answers the above question affirmatively under the assumption that the generating Hamiltonians, $H_i$, converge in $L^{(1, \infty)}$ topology.  The $L^{(1, \infty)}$ (or Hofer \cite{H, HZ}) norm of a Hamiltonian $K$ is given by:  $ \Vert K \Vert_{(1,\infty)} = \int_0^1 (\max_x K(t,x) - \min_x K(t,x))\,dt.$  Here is our main result:
  \begin{thm} (Main Theorem) \label{Main Theorem}
   			Consider a smooth isotopy, $\phi^t \text{ }, t \in [0, 1]$, which is a $C^0$ limit of Hamiltonian paths, $\phi_{H_i}^{t} \text{ }, t \in [0, 1]$.  If the sequence of Hamiltonians $H_i$ has a $L^{(1, \infty)}$ limit $H:[0,1] \times M \rightarrow \mathbb{R}$, then 
   	
   	   \begin{enumerate}
   	      \item The $L^{(1, \infty)}$ equivalence class of $H$ has a smooth representative.
   	      \item $\phi^t$ is a Hamiltonian path and the smooth representative of $H$ is its Hamiltonian function.
   	    \end{enumerate}
   \end{thm}
   
    The methods employed in this paper are very different than those used by Lalonde, McDuff, and Polterovich in \cite{LMP}.  We use M\"uller and Oh's theory of topological Hamiltonian paths \cite{OM}.  The main tool used in this paper is a theorem  from \cite{BS} on uniqueness of generators for topological Hamiltonian paths and its corollaries.  An analogous result has been proven for contact isotopies in Section 13.2 of \cite{MuSp}.  In Section \ref{Topological Hamiltonian Paths}, we will review the theory of topological Hamiltonian paths and reformulate Theorem \ref{Main Theorem} as a question about topological Hamiltonian paths.  In Section \ref{Proofs}, we prove our main result.
   
   \section{Topological Hamiltonian Paths} \label{Topological Hamiltonian Paths}
    
    \begin{defn}(Topological Hamiltonian Paths \cite{OM})\\ \label{Topological Hamiltonian Paths Definition}
    An isotopy $\phi_H^t : M \rightarrow M \text{ } (0\leq t \leq 1)$ is called a topological Hamiltonian path generated by $H \in L^{(1,\infty)}([0,1] \times M)$ if there exists a sequence of smooth Hamiltonian paths, $\phi^t_{H_i} \text{ } (0\leq t \leq 1)$ such that $ \phi^t = (C^0) \lim_{i \to \infty} \phi^t_{H_i}$  and $ H = (L^{(1,\infty)}) \lim_{i \to \infty} H_i. $  Denote by $\PHameo(M)$ the set of all topological Hamiltonian paths. 
    
    \end{defn}
     
   Observe that Theorem \ref{Main Theorem} says that if $\phi^t_H \in PHameo(M)$ happens to be smooth then it is a smooth Hamiltonian path and $H$ is its (smooth) Hamiltonian.  Oh and M\"uller showed that topological Hamiltonian paths satisfy the following basic properties.
   
   \begin{thm}(\cite{OM}) \label{Basic Properties}
   Suppose $\phi^t_H \text{, } \phi^t_K \in PHameo(M)$.  Then,
   
   \begin{enumerate}
   
      \item $\phi^t_{H \# K} := \phi^t_H \circ \phi^t_K$ is a topological Hamiltonian path and it is generated by  $H \# K(t,x):= H(t,x) + K(t,(\phi^{t}_{H})^{-1}(x))$.  
      \item $ \phi^t_{\overline{H}} := (\phi^t_H)^{-1}$ is a topological Hamiltonian path and it is generated by $\overline{H}(t,x):= -H(t,\phi^t_{H}(x))$.
      \item If $\alpha : [0,1] \rightarrow [0,1]$ is a $C^1$ map then $\phi^t_{H^{\alpha}} := \phi^{\alpha(t)}_H $ is a topological Hamiltonian path and it is generated by $H^{\alpha}(t,x) : = \alpha'(t)  H(\alpha(t),x)$.
   
   \end{enumerate}

   \end{thm}

   \noindent Note that (1) and (2) in Theorem \ref{Basic Properties} imply that $\displaystyle PHameo(M)$ is a group.  

     \begin{rem} \label{R: rescaling of intervals}
     In Definition \ref{Topological Hamiltonian Paths Definition}, we require that the time variable $t$ take values in $[0,1]$.  It can  easily be seen that the interval $[0,1]$ can be replaced with any interval of the form $[0,a]$, for any $a \in \mathbb{R}$.  
   \end{rem}

     Note that because topological Hamiltonian paths are defined via limits of smooth Hamiltonian paths, it is not clear that a given topological Hamiltonian path has a unique $L^{(1, \infty)}$ generator.  This question was raised in \cite{OM}. In \cite{V}, Viterbo showed that if one replaces the assumption $H = (L^{(1,\infty)}) \lim_{i \to \infty} H_i$ in Definition \ref{Topological Hamiltonian Paths Definition} with $H = (C^0) \lim_{i \to \infty} H_i$ then uniqueness holds on closed manifolds.  Oh extended Viterbo's result to open manifolds \cite{Oh}.  The general case, where it is assumed that $H = (L^{(1,\infty)}) \lim_{i \to \infty} H_i$, was settled in \cite{BS}.

    \begin{thm} (Uniqueness of Generators)\label{Uniqueness of Generators}
         Every topological Hamiltonian path has a unique $L^{(1, \infty)}$ generator.
    \end{thm}
    
    The main tool used in proving Theorem \ref{Main Theorem} is the following generalization of Theorem \ref{Uniqueness of Generators}. This is Theorem 11 from \cite{BS}.
    
    \begin{thm} (Local Uniqueness of Generators, see Thm. 11 in \cite{BS}) \label{Local Uniqueness of Generators}
    Suppose $\phi^t_H \in PHameo(M)$ coincides with the identity map on some open subset $ U \subset M $, i.e. $ \phi^t_H (x) = x \text{ } \forall (t,x) \in [0,1] \times U$. Then for almost all $ t \in [0,1] $ the restriction $ H(t,\cdot)|_{U} $ is a constant function depending on t.     
    \end{thm}
    
    \begin{rem} \label{Remark to Local Uniqueness of Generators}
    The above statement remains true if the the interval $[0,1]$ is replaced by $[0,T]$, for some $T \leq 1$.  This follows from property (3) in Theorem \ref{Basic Properties}.
     
    \end{rem}
   
\section{Proofs} \label{Proofs}
 		 
 		 We will need the following lemma:
 		 
 		 \begin{lem} \label{time shift lemma}
	       Let $\phi^t_H$, $t\in[0,1]$, denote a topological Hamiltonian path.  For each fixed $s \in [0,1), \text{} \phi_K^t := \phi_H^{t+s} \circ (\phi_H^s)^{-1}:[0,1-s] \times M \rightarrow M,$ is a topological Hamiltonian path with generator $K(t,x): [0,1-s] \times M \rightarrow \mathbb{R}$, where $K(t,x) := H(t+s,x)$.  Note that $\phi_K^t$ is defined on the interval $[0, 1-s]$; see Remark \ref{R: rescaling of intervals}.
	   \end{lem}
	   
	   \begin{proof}
	      Fix $s \in [0, 1)$.
	      We must show that there exist Hamiltonians $K_i: [0, 1-s] \times M \rightarrow \mathbb{R}$  such that $(L^{(1, \infty)}) \lim K_i = K$ and $(C^0) \lim \phi^t_{K_i} = \phi_K^t.$  
	      
	      We know that  $\phi^t_H$ is a topological Hamiltonian path.  Hence, there exists a sequence of smooth Hamiltonians $H_i: [0,1] \times M \rightarrow \mathbb{R}$ such that $(L^{(1, \infty)}) \lim H_i = H$ and $(C^0) \lim \phi^t_{H_i} = \phi_{H}^t$.  
	      
	      Let $\phi^t_{K_i} := \phi_{H_i}^{t+s} \circ (\phi_{H_i}^s)^{-1}:[0,1-s] \times M \rightarrow M$.  The smooth Hamiltonian path  $\phi^t_{K_i}$ is generated by the Hamiltonian $K_i(t,x) := H_i(t+s,x):[0, 1-s] \times M \rightarrow \mathbb{R}.$  Now, the fact that $(C^0) \lim \phi^t_{H_i} = \phi_{H}^t$ implies that $(C^0) \lim \phi^t_{K_i} = \phi_{K}^t$ and the fact that $(L^{(1, \infty)}) \lim H_i = H$ implies that $(L^{(1, \infty)}) \lim K_i = K.$

	   \end{proof}

     \begin{thm}
			Let $\phi^t_H$ , $t\in[0,1]$, denote a topological Hamiltonian path. If $\phi^t_H$ is smooth then it is a smooth Hamiltonian path. 
	   \end{thm}
	   
	   \begin{proof}
	      Let $X$ denote the time dependent vector field obtained by differentiating $\phi^t_H$ with respect to time, i.e. $X_t( \phi^t_H(x)) = \frac{d}{dt} \phi^t_H(x)$.  Let $ \alpha_t = \omega(X_t,\cdot).$  We know, by the Gromov-Eliashberg $C^0$-rigidity theorem , that $\phi^t_H$ is a symplectic isotopy and hence $X$ is a symplectic vector field, i.e.  $ \alpha_t$ is a closed one form $\forall t\in [0,1]$.  We will break the proof down to several steps.    
	  \medskip
	    
	  \noindent \textbf{Step 1:}  For every point $p \in M$ there exist a neighborhood of it $W_p$, a real number $T_p \in (0,1]$, and a smooth function $G: [0,1] \times M \rightarrow \mathbb{R}$ such that:
	   \begin{enumerate}
	   \item $\alpha_t|_{W_p} = dG(t,\cdot)|_{W_p}$ for all $t \in [0,T_p].$
	   \item $H(t,\cdot)|_{W_p} = G(t, \cdot)|_{W_p} + c(t)$ for almost every $t\in [0, T_p]$, where $c(t)$ is a $L^1$ function of $t$.  
	   \end{enumerate} 
	   \noindent \textbf{Proof of Step 1:} Take any point $p \in M$. Because $\alpha_t$ is closed, there exists a smooth function $G: [0,1] \times M \rightarrow \mathbb{R}$ and a neighborhood $U$ of the point $p$ such that $\alpha_t|_{U} = dG(t, \cdot)|_{U}$.  Because the family of one forms $\alpha_t$ varies smoothly with $t$, we can pick the function $G$ so that it is smooth in $t$.
	      
	      Let $\phi^t_G$ denote the Hamiltonian flow of $G$.  Clearly, the vector field of $\phi^t_G$ coincides with $X$ on $U$.  Hence, the two flows $\phi^t_H$ and $\phi^t_G$ coincide on a neighborhood of $p$, say $V_p$, for a short time period.  In other words, $\exists r_p  \in (0,1]$ (depending on $p$) and $\exists \; V_p \subset U$ such that
	    \[ \phi^t_H(x)=\phi^t_G(x) \text{, } \forall x \in V_p \text{ and } \forall t\in[0,r_p]. \]
 	    
	   By parts (2) and (3) of Theorem \ref{Basic Properties}, the composition $(\phi_G^t)^{-1} \circ \phi^t_H$ is a topological Hamiltonian path and it is generated by $ -G(t,\phi^t_G(x)) + H(t, \phi^t_G(x))$. Clearly, $(\phi_G^t)^{-1} \circ \phi^t_H(x) = x \text{ }  \forall x \in V_p \text{ and } \forall t\in[0,r_p]$. Thus, Theorem ~\ref{Local Uniqueness of Generators} and Remark ~\ref{Remark to Local Uniqueness of Generators} imply that $\forall x \in V_p$ and for almost all $ t \in [0,r_p]$ we have
	   \begin{equation}\label{H and G coincide locally}
	    H(t, \phi^t_G(x)) = G(t,\phi^t_G(x)) + c(t)
	   \end{equation}
	   where $c(t)$ is that it is a $L^1$ function of $t$.  Now, there exist a small neighborhood of $p$, say $W_p \subset V_p$, and a small real number $T_p \in (0,1]$ such that $(\phi_G^t)^{-1}(W_p)\subset V_p$.  We may assume that $T_p \leq r_p$.   For any point $x \in W_p$, we know that $(\phi_G^t)^{-1}(x) \in V_p$ for all $t \in [0, T_p]$.  Hence by (\ref{H and G coincide locally}), $H(t,x) = G(t,x) + c(t)$ for all $x \in W_p$ and almost every $t \in [0, T_p]$. 
	   
	   \medskip 
	   
	   \noindent \textbf{Step 2:} There exist a finite cover $\{W_i\}_{i=1}^{i=k}$ of $M$, smooth functions $\{G_i: [0,1] \times M \rightarrow \mathbb{R}\}_{i=1}^{i=k}$ and $T \in (0,1]$ such that:
	   \begin{enumerate}
	   \item $\alpha_t|_{W_i} = dG_i(t,\cdot)|_{W_i}$ for all $t \in [0,T].$
	   \item $H(t,\cdot)|_{W_i} = G_i(t, \cdot)|_{W_i} + c_i(t)$ for every $t\in [0, T] \setminus A$, where $c_i(t)$ is a $L^1$ function of $t$ and $A$ is a measure zero subset of $[0,T]$.  
	   \end{enumerate} 
	   
	  \noindent \textbf{Proof of Step 2:} Let $\{W_p: p\in M \}$ denote the collection of sets obtained in Step 1.  By compactness of $M$ we can pick a finite subcover $\{W_i\}_{i=1}^{i=k}$ with associated time intervals $[0, T_i]$ and functions $G_i$.  Let $T = \min_i\{T_i\}$.  Then it is clear that $\alpha_t|_{W_i} = dG_i(t,\cdot)|_{W_i}$ for all $t \in [0,T].$  Now, for each $i$ there exists a set $A_i$ of measure zero such that for all $t \in [0, T] \setminus A_i$ we have $H(t,\cdot)|_{W_i} = G_i(t, \cdot)|_{W_i} + c_i(t)$.  Let $A = \bigcup_{i=1}^{k} A_i$.  Then $A$ has measure zero and for every $t \in [0,T] \setminus A$ we have $H(t,\cdot)|_{W_i} = G_i(t, \cdot)|_{W_i} + c_i(t)$.  Note that for all $t \in [0,T] \setminus A$, and hence for almost every $t \in [0, T]$, the function $H(t, \cdot) : M \rightarrow \mathbb{R}$ is smooth on $M$.
	   
	    \medskip
	    
	    \noindent \textbf{Step 3:} For almost every $t \in [0,T]$ the one form $\alpha_t$ is exact.
	   
	    \noindent \textbf{Proof of Step 3:} By Step 2, $\alpha_t|_{W_i} = dG_i(t, \cdot)|_{W_i} \text{ } \forall t\in[0,T]$ and $H(t,x)|_{W_i} = G_i(t,x)|_{W_i} + c_i(t)$ for every $t \in [0, T] \setminus A$.  Clearly, this implies that $\alpha_t|_{W_i} = dH(t,\cdot)|_{W_i}$ for every $t\in [0,T] \setminus A$.  Because the sets $W_i$ cover $M$ we conclude that $\alpha_t$ is exact for every $t \in [0,T]\setminus A$.
	      
	    \medskip
	    
	    \noindent \textbf{Step 4:} For every $t \in [0,T]$ the one form $\alpha_t$ is exact.  Hence, the smooth isotopy $\phi^t_H \text{ } 0 \leq t \leq T$ is a smooth Hamiltonian path.
	   
	   \noindent \textbf{Proof of Step 4:} Let $\gamma : S^1 \rightarrow M$ denote a smooth loop.  Let $f(t) := \int_{\gamma}{ \alpha_t}$.  Then, $f(t)$ is a smooth function of $t$ because the family of one forms $\alpha_t$ depends smoothly on $t$.  Now, $\alpha_t$ is exact for almost every $t \in [0,T]$, hence $f(t) = 0$ for almost every $t\in [0, T]$.  Smoothness of $f$ implies that $f(t)=0$ for every $t \in [0, T].$  Therefore, for every $t \in [0, T]$ and any loop $\gamma$ in $M \text{ } \int_{\gamma}{ \alpha(t)}=0$, which means that $\alpha_t$ is exact for every $t\in [0, T]$.
	    
	   The above implies that for each $t\in [0,T]$ there exists a smooth function  $F(t, \cdot): M \rightarrow \mathbb{R}$ such that $\alpha_t = dF(t,\cdot)$. The functions $F(t, \cdot)$ can be picked such that they depend smoothly on $t$.  This, by definition, means that the isotopy $\phi^t_H \text{ } 0\leq t \leq T$ is a smooth Hamiltonian path, and in fact by the uniqueness theorem $F=H$ in $L^{(1, \infty)}([0,T] \times M)$.

	   \medskip
	      
	   \noindent \textbf{Step 5:} The smooth isotopy $\phi^t_H \text{ } 0 \leq t \leq 1$ is a smooth Hamiltonian path.
	    
	   \noindent \textbf{Proof of Step 5:} Let $s= \sup\{ r \in [0,1] : \phi^t_H \in Ham(M)  \; \forall t \in [0, r] \}$.  By Step 4, $s \geq T$ and hence $s$ is positive.
	    First, note that $\phi^s_H$ is Hamiltonian.  This is because $\phi^t_H \in Ham(M)$ for all $t \in [0, s)$ which means that $\alpha_t$ is exact for all $t \in [0, s)$.  The argument used in Step 4 shows that $\alpha_s$ is exact. 
	    	    
	    If $s=1$ we are done.  For a contradiction, assume $s < 1$.  Let $\phi_{H_s}^t := \phi_{H}^{t+s} \circ (\phi_{H}^s)^{-1}:[0,1-s] \times M \rightarrow M$.  By lemma \ref{time shift lemma}, $\phi_{H_s}^t$ is a topological Hamiltonian path (defined on [0, 1-s]) and it is generated by $H_s(t,x) = H(t+s,x)$.  Now, $\phi_{H_s}^t$ is a smooth topological Hamiltonian path.  Hence, by Steps 1 to 4 we can conclude that there exists $T' \in (0, 1-s]$ such $\phi_{H_s}^t \in Ham(M)$ for all $t \in [0, T']$.  Clearly, this implies that $\phi_{H}^{t+s}\in Ham(M)$ for all $t \in [0,T']$.  Thus, $s = sup\{ r \in [0,1] : \phi^t_H \in Ham(M)  \; \forall t \in [0, r] \} \geq T'+s$  which is a contradiction.  Hence, $s=1$ and we are done. \end{proof}
	   
	   \begin{proof} (Proof of the Main Theorem)
	   By Definition \ref{Topological Hamiltonian Paths Definition}, $\phi^t$ is a topological Hamiltonian path generated, in the sense of Definition \ref{Topological Hamiltonian Paths Definition}, by $H$.  Since $\phi^t$ is smooth the previous theorem implies that it is a smooth Hamiltonian path generated, in the usual sense, by a smooth Hamiltonian $F$.  Theorem \ref{Uniqueness of Generators} on uniqueness of generators implies that $H=F$ as $L^{(1, \infty)}$ functions. \end{proof}
	   
	   \subsubsection*{Acknowledgments}
	    I am thankful to Denis Auroux, Lev Buhovsky, and Beno\^it Jubin for helpful comments and conversations.  I would also like to thank my advisor, Alan Weinstein, for his suggestions, support and encouragement.

\end{document}